\documentclass[12pt,reqno]{amsart}
\usepackage{amsmath, amsthm, amscd, amsfonts, amssymb, graphicx, color}
\usepackage[bookmarksnumbered, colorlinks, plainpages]{hyperref}

\newtheorem{theorem}{Theorem}[section]
\newtheorem{lemma}[theorem]{Lemma}
\newtheorem{proposition}[theorem]{Proposition}
\newtheorem{corollary}[theorem]{Corollary}
\theoremstyle{definition}

\theoremstyle{remark}
\newtheorem{remark}[theorem]{Remark}
\numberwithin{equation}{section}

\begin{document}
\setcounter{page}{1}

\title[Inequalities for Gauss Hypergeometric Functions]{On some Analytic Inequalities for Gauss Hypergeometric Functions via Gr\"{u}ss Discrete Inequality}

\author[M. Ra\"{\i}ssouli, M. Chergui]{Mustapha Ra\"{\i}ssouli$^{1}$ and Mohamed Chergui$^2$}

\address{$^{1}$ Department of Mathematics, Science Faculty, Moulay Ismail University, Meknes, Morocco.}

\address{$^{2}$ Department of Mathematics, CRMEF-RSK, EREAM Team, LaREAMI-Lab, Kenitra, Morocco.}

\email{\textcolor[rgb]{0.00,0.00,0.84}{raissouli.mustapha@gmail.com}}
\email{\textcolor[rgb]{0.00,0.00,0.84}{chergui\_m@yahoo.fr}}



\date{Received: xxxxxx; Revised: yyyyyy; Accepted: zzzzzz.}

\begin{abstract}
Recently, many researchers devoted their attention to study the extensions of the gamma and beta functions. In the present work, we focus on investigating some approximations for a class of Gauss hypergeometric functions by exploiting Gr\"{u}ss discrete inequality.
\end{abstract}

\maketitle

\section{Introduction}

 The gamma and beta functions are widely applied in modeling many phenomena in different fields of pure and applied sciences. Many generalizations and extensions concerning these special functions have been carried out, as we can find for instance in \cite{Chaudhry et al1,Chaudhry et al,Ozergin et al,RA1,RA2,Slater} and in the related references cited therein.

 In this preliminary section, we will present a brief overview about some classes of Gauss hypergeometric functions that will be used in the current study. We begin by recalling the Gauss hypergeometric function (GHF) that was defined in \cite[p.1]{Slater} by
\begin{gather}\label{GHFexpansion1}
   _2F_1(a, b; c;z) := \sum_{n=0}^{\infty} \frac{(a)_n (b)_n}{(c)_n}\,\frac{z^n}{n!},\;\; \Re e(a)>0,\; \Re e(b)>0,\; \Re e(c)>0,
\end{gather}
provided that this series is convergent. Here, the notation $(\lambda)_n$, for $\Re e(\lambda)>0$, refers to the Pochhammer symbol defined by
$$ (\lambda)_n:=\lambda(\lambda+1)...(\lambda+n-1)=\frac{\Gamma(\lambda+n)}{\Gamma(\lambda)} \text{ and } (\lambda)_0 =1,$$
where $\Gamma(x):=\int_0^\infty t^{x-1}e^{-t}dt,\; \Re e(x)>0,$ is the standard gamma function. \\
The (GHF) can be written in terms of the classical beta function as follows \cite{Chaudhry et al1}
\begin{gather}\label{GHFexpansion2}
    _2F_1(a, b; c;z)= \frac{1}{B(b,c-b)}\sum_{n=0}^{\infty} (a)_n B(b+n,c-b)\,\frac{z^n}{n!},
\end{gather}
provided that $0<\Re e(b)<\Re e(c)$, where
$B(x,y):=\int_0^1 t^{x-1}\,(1-t)^{y-1}\,dt$, for $\Re e(x)>0,\Re e(y)>0$,
is the standard beta function. The formula \eqref{GHFexpansion2} was deduced from the following integral representation, \cite[p.20]{Slater}

\begin{gather}\label{integralRepresentationGHF1}
     _2F_1(a, b; c;z) = \frac{1}{B(b,c-b)}\int_0^1 t^{b-1}(1-t)^{c-b-1} (1-zt)^{-a}\,dt,\;\; |z|\leq1.
\end{gather}

The confluent hypergeometric function (CHF), also called Kummer's function, was defined in \cite{Macdonald AD} by the following formula
\begin{gather}\label{CHF}
   _1F_1 (b; c; z) := \sum_{n=0}^{\infty} \frac{(b)_n}{(c)_n}\,\frac{z^n}{n!},\;\; \Re e(b)>0,\; \Re e(c)>0,
\end{gather}
provided that this series is convergent. It can be written in terms of the classical beta function as follows \cite[p.504]{Abramowitz}
\begin{gather}\label{expansionCHF}
    _1F_1 (b; c; z)= \frac{1}{B(b,c-b)}\sum_{n=0}^{\infty} B(b+n,c-b)\,\frac{z^n}{n!},\;\; 0<\Re e(b)<\Re e(c).
\end{gather}
In integral form we have, \cite{Ozergin et al}
\begin{gather}\label{CHFinIntegral}
    _1F_1 (b; c; z)= \frac{1}{B(b,c-b)}\int_0^1 u^{b-1}(1-u)^{c-b-1}e^{zu}\,du.
\end{gather}
The definition of the confluent hypergeometric function given by \eqref{CHF} permits \"{O}zergin et al. to introduce in \cite{Ozergin et al} the following
generalized beta function
\begin{gather}\label{BetaGeneralized}
   B_p^{(\alpha,\beta)}(x,y) := \int_0^1 t^{x-1}(1-t)^{y-1}\, _1F_1 \left(\alpha; \beta; \frac{-p}{t(1-t)}\right)\,dt,
\end{gather}
where $\Re e(\alpha)>0,\Re e(\beta)>0,\Re e(p)\geq0$. It is obvious that $B_0^{(\alpha,\beta)}(x,y)=B(x,y)$. For the seek of simplicity, we need to normalize \eqref{BetaGeneralized} by setting
\begin{equation}\label{N}
 B^{(\alpha,\beta;p)}(x,y):=\frac{1}{B(x,y)}\int_0^1 t^{x-1}(1-t)^{y-1}\, _1F_1 \left(\alpha; \beta; \frac{-p}{t(1-t)}\right)\,dt.
\end{equation}

Based on \eqref{BetaGeneralized}, \"{O}zergin et al. defined in \cite{Ozergin et al} the generalized Gauss hypergeometric function (GGHF) and
the generalized confluent hypergeometric function (GCHF), respectively, as follows
\begin{gather}\label{GGHF}
   _2F_1^{(\alpha,\beta;p)} (a, b; c; z) :=\frac{1}{B(b,c-b)}\sum_{n=0}^\infty (a)_n\,B_p^{(\alpha,\beta)}(b+n,c-b)\,\frac{z^n}{n!},
\end{gather}
\begin{gather}\label{GCHF}
   _1F_1^{(\alpha,\beta;p)} (b; c; z) :=\frac{1}{B(b,c-b)}\sum_{n=0}^\infty B_p^{(\alpha,\beta)}(b+n,c-b)\,\frac{z^n}{n!}.
\end{gather}
It is clear that
\begin{gather}\label{eq1}
_2F_1^{(\alpha,\beta;0)} (a, b; c; z)=\;_2F_1(a, b; c;z) \text{ and }  _1F_1^{(\alpha,\beta;0)} (b; c; z)=\;_1F_1 (b; c; z).
\end{gather}
In \cite{Ozergin et al}, the authors gave integral representations of (GGHF) and (GCHF) as follows
\begin{multline}\label{eq15}
 _2F_1^{(\alpha,\beta;p)} (a,b; c; z)\\
    =\frac{1}{B(b,c-b)}\int_0^1 t^{b-1} (1-t)^{c-b-1}\,(1-zt)^{-a}\,_1F_1\left(\alpha;\beta; \frac{-p}{t(1-t)}\right)\,dt,\;\; |z|\leq1,
\end{multline}
\begin{multline}\label{eq2}
    _1F_1^{(\alpha,\beta;p)} (b; c; z)\\
    =\frac{1}{B(b,c-b)}\int_0^1 t^{b-1} (1-t)^{c-b-1}\,e^{zt}\,_1F_1\left(\alpha;\beta; \frac{-p}{t(1-t)}\right)\,dt.
\end{multline}

The following remark will be of interest in the sequel.

\begin{remark}\label{remA}
All the previous series as well as their related improper integrals are uniformly convergent. This implies that, whenever we have an infinite summation with one of these improper integrals then the order of the integral with the summation can be interchanged. We can also differentiate and take the limit under these infinite summations or improper integrals.
\end{remark}
The fundamental target of this manuscript is to investigate some approximations for the previous generalized hypergeometric functions by the use of Gr\"{u}ss discrete inequality.

\section{The main results}

We start this section by giving some upper bounds of the previous hypergeometric functions in terms of simple analytic functions.

\begin{proposition}
Let $a>0, a\neq 1$, $b\geq1, c\geq b+1,\alpha\geq1,\beta\geq\alpha+1, p>0$, $z\in\mathbb{R}^*,$ with $|z|<1$ and $t\in(0,1)$. We have the following inequalities
\begin{gather}\label{IneqByAndreief2}
_1F_1 (b; c; z) \le  \frac{\lambda_{b-1,c-b-1}}{ B(b,c-b)}\, \times \frac{e^z-1}{z},
\end{gather}
\begin{gather}\label{IneqByAndreief2proof}
\,_1F_1 \left(\alpha; \beta; \frac{-p}{t(1-t)}\right)\le\frac{\lambda_{\alpha-1,\beta-\alpha-1}}{4\,p\,B(\alpha,\beta-\alpha)},
\end{gather}
\begin{gather} \label{GBetabounds}
B_p^{(\alpha,\beta)}(b,c-b) \le \theta^{\alpha,\beta}_p(b,c),
\end{gather}
\begin{gather}\label{F2}
_2F_1 (a,b; c; z) \le  \frac{\lambda_{b-1,c-b-1}}{ B(b,c-b)}\,\times\frac{1-(1-z)^{1-a}}{(1-a)z},
\end{gather}
where we set
\begin{gather*}
\lambda_{x,y}:=\dfrac{x^x\,y^y}{(x+y)^{x+y}}\;\; \mbox{for}\; x,y\geq0 \text{ and }
\theta^{\alpha,\beta}_p(b,c):=\dfrac{\lambda_{b,c-b}\;\lambda_{\alpha-1,\beta-\alpha-1}}{4\,p\, B(\alpha,\beta-\alpha)}.
\end{gather*}
\end{proposition}
\begin{proof}
One can easily check that
\begin{gather*}
\displaystyle\sup_{\substack{t \in (0,1)}}t^x\,(1-t)^y = \frac{x^{x}y^{y}}{(x+y)^{x+y}}:=\lambda_{x,y} \text{ for all } x, y\geq0,
\end{gather*}
with the usual convention $0^0=1$. Using this, we get
\begin{eqnarray*}
 \,_1F_1 (b; c; z) &=& \frac{1}{ B(b,c-b)}\,\int_0^1 t^{b-1}\,(1-t)^{c-b-1}\,e^{zt}\;dt\\
   &\le & \frac{\lambda_{b-1,c-b-1}}{ B(b,c-b)}\,\int_0^1e^{zt}\;dt \\
   &=& \frac{\lambda_{b-1,c-b-1}}{ B(b,c-b)}\, \times \frac{e^z-1}{z},
\end{eqnarray*}
hence \eqref{IneqByAndreief2}. It follows that we have
\begin{gather}\label{XX}
\,_1F_1 \left(\alpha; \beta; \frac{-p}{t(1-t)}\right)\le  \frac{\lambda_{\alpha-1,\beta-\alpha-1}}{ B(\alpha,\beta-\alpha)}
\times \frac{1-e^{\frac{-p}{t(1-t)}}}{p}\,t(1-t).
\end{gather}
It is obvious that $\frac{1-e^{\frac{-p}{t(1-t)}}}{p}\le \frac{1}{p}$ and $t(1-t)\le \frac{1}{4}$ for all $t\in (0,1)$, and so \eqref{IneqByAndreief2proof} follows. By \eqref{BetaGeneralized} and \eqref{IneqByAndreief2proof}, with the fact that  $t^{b-1}\,(1-t)^{c-b-1}\le \lambda_{b,c-b}$ for all $t\in(0,1)$, we get \eqref{GBetabounds}. Finally, \eqref{F2} can be proved in a similar manner as \eqref{IneqByAndreief2}.
\end{proof}

Using \eqref{eq15} and \eqref{eq2}, with the help of \eqref{IneqByAndreief2proof}, we deduce the following corollary.

\begin{corollary}
With the same conditions as in the previous proposition, we have
$$_1F_1^{(\alpha,\beta;p)}(b;c;z)\leq\frac{\lambda_{\alpha-1,\beta-\alpha-1}}{4 \,p\, B(\alpha,\beta-\alpha)}\;_1F_1(b;c;z),$$
$$_2F_1^{(\alpha,\beta;p)}(a,b;c;z)\leq\frac{\lambda_{\alpha-1,\beta-\alpha-1}}{4 \,p\, B(\alpha,\beta-\alpha)}\;_2F_1(a,b;c;z).$$
\end{corollary}

As already pointed out before, the aim of this section is to investigate some inequalities involving the functions (GCHF) and (GHF) by using the Gr\"{u}ss discrete inequality \cite{Cerone and Dragomir} which we recall in the following.

\begin{lemma}\label{lemCD}
Let $(x_k),(y_k),(m_k), 1\leq k\leq n,$ be three families of reals numbers such that $m_k\geq0$ for any $k=1,2,...,n$. Then we have
\begin{multline}\label{GD}
 \left\vert \sum_{k=1}^{n} m_k \times \sum_{k=1}^{n} m_k\,x_k\,y_k - \sum_{k=0}^{n} m_k\,x_k \times \sum_{k=1}^{n} m_k\,y_k\right \vert \le \\
 \frac{1}{4}  \left(\sum_{k=1}^{n} m_k\right)^2 (\Gamma-\gamma)(\Phi-\varphi),
\end{multline}
provided $ \gamma \le x_k \le \Gamma$, $ \varphi\le y_k \le \Phi$ and $m_k\geq0$, for all $k \in \{1, \dots,n\}$. Further, the constant $\frac{1}{4}$ is sharp.
\end{lemma}

We need to introduce an appropriate notation in order to simplify the writing of some inequalities throughout this paper.
With the above notations, we set
\begin{equation}\label{R1}
_1R_1^{(\alpha,\beta;p)}(b;c;z):=\frac{_1F_1^{(\alpha,\beta;p)}(b;c;z)}{B^{(\alpha,\beta;p)}(b,c-b)},
\end{equation}
\begin{equation}\label{R2}
_2R_1^{(\alpha,\beta;p)}(a,b;c;z):=\frac{_2F_1^{(\alpha,\beta;p)}(a,b;c;z)}{B^{(\alpha,\beta;p)}(b,c-b)}.
\end{equation}
It is easy to see that
\begin{equation}\label{p=0}
_1R_1^{(\alpha,\beta;0)}(b;c;z)=\;_1F_1(b;c;z),\;\;\;
_2R_1^{(\alpha,\beta;0)}(a,b;c;z)=\;_2F_1(a,b;c;z).
\end{equation}

Now, we are in the position to state our first main result recited in the following.

\begin{theorem}\label{thA}
Let $0<b< c$, $\alpha,\beta>0$ and $z,z_0\in \mathbb{R}$ be such that $0<z\leq z_0$. We have the following inequalities
\begin{multline}\label{ineg0}
  \bigg\vert (e^{z_0}-1)\,\Big(\;_1R_1^{(\alpha,\beta;p)}(b;c;z)-1\Big)\\
  -\big(e^z-1\big) \Big(\;_1R_1^{(\alpha,\beta;p)}(b;c;z_0)-1\Big)\, \bigg\vert \le
  \frac{(e^{z_0}-1)^2}{4}\,\frac{z}{z_0},
\end{multline}
and
\begin{gather}\label{ineqGruss01}
   \bigg\vert \,_1R_1^{(\alpha,\beta;p)}(b;c;z)-e^{z-z_0}\,_1R_1^{(\alpha,\beta;p)}(b;c;z_0)\bigg\vert \le
  \frac{e^{z_0}}{4}.
\end{gather}
\end{theorem}
\begin{proof}
We consider the following sequences,
$$x_k= B_p^{(\alpha,\beta)}(b+k,c-b),\; y_k=\Big(\frac{z}{z_0}\Big)^k \text{ and } m_k=\frac{z^k_0}{k!}\; \text{ for } k\geq 0.$$
Since $0<z\leq z_0$ then for all $k\geq1$, $0 \le y_k \le\dfrac{z}{z_0}$. Otherwise, it is easy to see that $(x_k)$ is decreasing and so we can write
\begin{gather}\label{M}
\forall k\geq1\;\;\;\;\; 0\leq x_k\le B_p^{(\alpha,\beta)}(b,c-b).
\end{gather}
Applying the weighted Gr\"{u}ss inequality \eqref{GD} for the bounded sequences $(x_k)$ and $(y_k)$, with the weights $(m_k)$, we get
\begin{multline}\label{ineq01proof01}
    \left\vert \sum_{k=1}^{n} m_k \times \sum_{k=1}^{n} m_k\,x_k\,y_k -\sum_{k=1}^{n} m_k\,x_k \times \sum_{k=1}^{n} m_k\,y_k\right \vert\\
    \le \frac{1}{4}\frac{z}{z_0} \left(\sum_{k=1}^{n} m_k\right)^2 B_p^{(\alpha,\beta)}(b,c-b).
\end{multline}
In another part, we have
\begin{gather}\label{u01}
\sum_{k=1}^{\infty} m_k=\sum_{k=1}^{\infty} \frac{z_0^k}{k!}=e^{z_0}-1 \text{ and } \sum_{k=1}^{\infty} m_k\,y_k=  \sum_{k=1}^{\infty}\,\frac{z^k}{k!}=e^z-1,
\end{gather}
and from \eqref{GCHF} we get, respectively,
\begin{multline}\label{v01}
\sum_{k=1}^{\infty} m_k\,x_k\,y_k= \sum_{k=1}^{\infty} B_p^{(\alpha,\beta)}(b+k,c-b)\,\frac{z^k}{k!}= \\
B(b,c-b)\,_1F_1^{(\alpha,\beta;p)}(b;c;z)- B_p^{(\alpha,\beta)}(b,c-b),
\end{multline}
\begin{multline}\label{w01}
\sum_{k=1}^{\infty} m_k\,x_k=\sum_{k=1}^{\infty}B_p^{(\alpha,\beta)}(b+k,c-b) \, \frac{z^k_0}{k!}\\
=B(b,c-b)\,_1F_1^{(\alpha,\beta;p)}(b;c;z_0)- B_p^{(\alpha,\beta)}(b,c-b).
\end{multline}
Letting  $n\longrightarrow \infty$ in \eqref{ineq01proof01} and then substituting \eqref{u01},\eqref{v01} and \eqref{w01} therein, we obtain \eqref{ineg0} after a simple manipulation and by using \eqref{N} and \eqref{R1}.

Now, by considering the same previous sequences $(x_k),(y_k)$ and $(m_k)$ for $k$ from $0$ into $n$, we have $0\leq y_k\leq1$ and \eqref{M} persists for any $k\geq0$. Since \eqref{ineq01proof01} holds for $k=0,1,...,n$, we can deduce \eqref{ineqGruss01} in a similar manner after administering some little algebraic manipulations. The details are straightforward and therefore omitted here.
\end{proof}

Taking $p=0$ in the previous theorem and using the first relationship in \eqref{p=0} we immediately obtain the following corollary.

\begin{corollary}
Let $0<b<c$ and $z,z_0\in\mathbb{R}$ be such that $0<z\leq z_0$. Then we have
\begin{multline*}
\bigg\vert \big(e^{z_0}-1\big)\,\Big(\,_1F_1(b;c;z)- 1\Big) -\big(e^z-1\big)\,\Big(\,_1F_1(b;c;z_0)- 1\Big)\, \bigg\vert \le
\frac{(e^{z_0}-1)^2}{4}\,\frac{z}{z_0},
\end{multline*}
and
\begin{gather*}
\bigg\vert \,_1F_1(b;c;z) - e^{z-z_0} \,_1F_1(b;c;z_0) \bigg\vert \le  \frac{e^{z_0}}{4}.
\end{gather*}
\end{corollary}

Our second main result reads as follows.

\begin{theorem}\label{thI0}
Let $a>0, 0<b< c, \alpha>0,\beta>0, p\geq0$  and $0<z\leq z_0<1$. Then we have the following inequalities
\begin{multline}\label{ineqI0}
\bigg\vert \big((1-z_0)^{-a}-1\big)\, \left(\,_2R_1^{(\alpha,\beta;p)}(a,b;c;z)-1\right)\\
-\big((1-z)^{-a}-1\big)\, \left(\;_2R_1^{(\alpha,\beta;p)}(a,b;c;z_0)-1\right)\bigg\vert \le
\frac{z}{4}  \big((1-z_0)^{-a}-1\big)^2,
\end{multline}
\begin{equation}\label{ineqIa0}
\bigg\vert (1-z)^{a}\,_2R_1^{(\alpha,\beta;p)}(a,b;c;z)-(1-z_0)^{a}\,_2R_1^{(\alpha,\beta;p)}(a,b;c;z_0)\bigg\vert \le
\frac{z}{4}\frac{(1-z)^a}{(1-z_0)^{a}}.
\end{equation}
\end{theorem}
\begin{proof}
We consider the following sequences,
$$x_k=\Big(\frac{z}{z_0}\Big)^k,\; y_k= B_p^{(\alpha,\beta)}(b+k,c-b)\; \text{ and }\; m_k=(a)_k\frac{z_0^k}{k!},\;\;k\geq0.$$
As above, we have for all $k \ge 1$,
\begin{gather*}
0\le x_k \le \frac{z}{z_0} \text{ and } 0\le y_k \le B_p^{(\alpha,\beta)}(b,c-b).
\end{gather*}

Here we have
\begin{multline}\label{w0}
\sum_{k=1}^{\infty} m_k\,x_k=\sum_{k=1}^{\infty}(a)_k \, \frac{z^k}{k!} = (1-z)^{-a}-1,\\
\sum_{k=1}^{\infty} m_k=\sum_{k=1}^{\infty}(a)_k\frac{z_0^k}{k!}=(1-z_0)^{-a}-1,
\end{multline}
and from \eqref{GGHF} we get,
\begin{multline}\label{u0}
\sum_{k=1}^{\infty} m_k\,x_k\,y_k= \sum_{k=1}^{\infty} (a)_k B_p^{(\alpha,\beta)}(b+k,c-b)\,\frac{z^k}{k!}\\
= B(b,c-b)\,_2F_1^{(\alpha,\beta;p)}(a,b;c;z)-B_p^{(\alpha,\beta)}(b,c-b),
\end{multline}
\begin{multline}\label{v0}
\sum_{k=1}^{\infty} m_k\,y_k=  \sum_{k=1}^{\infty}(a)_k\,B_p^{(\alpha,\beta;p)}(b+k,c-b)\;\frac{z_0^k}{k!}\\
=B(b,c-b)\;_2F_1^{(\alpha,\beta;p)}(a,b;c;z_0)-B_p^{(\alpha,\beta)}(b,c-b).
\end{multline}
According to Gr\"{u}ss inequality \eqref{GD}, if we let $n\longrightarrow \infty$ and we then substitute \eqref{w0},\eqref{u0} and \eqref{v0}, we obtain \eqref{ineqI0} after simple algebraic operations with the use of \eqref{N} and \eqref{R2}.

We prove \eqref{ineqIa0} by similar arguments as in the the proof of Theorem \ref{thA}.
\end{proof}

If we take $p=0$ in Theorem \ref{thI0} and we use the second relationship in \eqref{p=0} we get the following result.

\begin{corollary}\label{thI10}
Let $a,b,c>0$ with $b<c$ and $0<z\leq z_0 $. Then we have
\begin{multline*}
\bigg\vert \big((1-z_0)^{-a}-1\big)\,\Big(\,_2F_1(a,b;c;z)-1\Big)-\big((1-z)^{-a}-1\big)\,\Big( \;_2F_1(a,b;c;z_0)-1\Big)\bigg\vert\\
\leq\frac{z}{4}\Big((1-z_0)^{-a}-1\Big)^2,
\end{multline*}
and
\begin{gather*}
\bigg\vert (1-z)^{a}\,_2F_1(a,b;c;z) -(1-z_0)^a\,_2F_1(a,b;c;z_0)\bigg\vert \le \frac{z}{4}\;\frac{(1-z)^a}{(1-z_0)^a}.
\end{gather*}
\end{corollary}

We also have the following main result.

\begin{theorem}\label{thB}
Let $0<b<c$, $\alpha>0,\beta>0,p\geq0$, $0\leq z_1\leq1$, $0\leq z_2\leq1$ and $z_3\geq0$. Then the following inequality holds:
\begin{gather*}
\left\vert \,_1F_1^{(\alpha,\beta;p)}\big(b;c;z_1z_2z_3\big)-e^{(z_2-1)z_3}\times\,_1F_1^{(\alpha,\beta;p)}\big(b;c;z_1z_3\big)\right \vert \le
     \frac{1}{4}e^{z_3}\;B^{(\alpha,\beta;p)}(b,c-b).
\end{gather*}
\end{theorem}
\begin{proof}
We now consider the following sequences,
\begin{gather*}
x_k= z_1^{k}\;B_p^{(\alpha,\beta)}(b+k,c-b),\; y_k= z_2^{k} \text{ and } m_k=\frac{z_3^{k}}{k!},\;\;k\geq0.
\end{gather*}
We have
\begin{gather*}
\forall k\geq0\;\;\; 0\leq x_k\le B_p^{(\alpha,\beta)}(b,c-b),\; 0\leq y_k\leq1.
\end{gather*}
The Gr\"{u}ss inequality \eqref{GD} yields
\begin{multline}\label{inequa1proof}
    \left\vert \sum_{k=0}^{n} m_k \times \sum_{k=0}^{n} m_k\,x_k\,y_k-\sum_{k=0}^{n} m_k\,x_k \times \sum_{k=0}^{n} m_k\,y_k\right \vert \le \\
     \frac{1}{4}\left(\sum_{k=0}^{n} m_k\right)^2\;B_p^{(\alpha,\beta)}(b,c-b).
\end{multline}
Computing, we have
\begin{gather*}
\sum_{k=0}^{\infty} m_k=\sum_{k=0}^{\infty} \frac{z_3^k}{k!}=e^{z_3} \text{ and } \sum_{k=0}^{\infty} m_k\,y_k=\sum_{k=0}^{\infty}\,\frac{(z_2z_3)^k}{k!}=e^{z_2z_3}.
\end{gather*}
and from \eqref{GCHF} we get, respectively,
\begin{gather*}
\sum_{k=0}^{\infty} m_k\,x_k\,y_k= \sum_{k=0}^{\infty} B_p^{(\alpha,\beta)}(b+k,c-b)\,\frac{(z_1z_2z_3)^k}{k!}=
B(b,c-b)\,_1F_1^{(\alpha,\beta;p)}\left(b;c;z_1z_2z_3\right),
\end{gather*}
\begin{gather*}
\sum_{k=0}^{\infty} m_k\,x_k=\sum_{k=0}^{\infty}B_p^{(\alpha,\beta)}(b+k,c-b)\,\frac{(z_1z_3)^k}{k!}= B(b,c-b)\,_1F_1^{(\alpha,\beta;p)}\left(b;c;z_1z_3\right).
\end{gather*}
Taking into account these latter expressions, letting  $n\longrightarrow \infty$ in \eqref{inequa1proof} and using \eqref{N} and \eqref{R1}, we then deduce the desired inequality.
\end{proof}

The particular case $p=0$ in Theorem \ref{thB} gives the following result.

\begin{corollary}
With the same assumptions as in Theorem \ref{thB}, we have
\begin{gather*}
\left\vert \,_1F_1\big(b;c;z_1z_2z_3\big)-e^{(z_1-1)z_3}\times\,_1F_1\big(b;c;z_1z_3\big)\right \vert \le
     \frac{1}{4}e^{z_3}.
\end{gather*}
\end{corollary}

Finally, we state an analogous result of Theorem \ref{thB} with $_2F_1^{(\alpha,\beta;p)}(a,b;c;z)$ instead of $_1F_1^{(\alpha,\beta;z)}(b;c;z)$.

\begin{theorem}\label{thC}
Let $a>0, 0<b< c,\alpha>0,\beta>0, p\geq0, 0\leq z_1\leq1, 0\leq z_2\leq1$ and $0\leq z_3<1$. Then we have
\begin{multline*}
\left\vert \,(1-z_2z_3)^a\,_2F_1^{(\alpha,\beta;p)}\big(a,b;c;z_1z_2z_3\big)-(1-z_3)^a\,_2F_1^{(\alpha,\beta;p)}\big(a,b;c;z_1z_3\big) \right \vert\\
\leq\frac{1}{4}\frac{(1-z_2z_3)^a}{(1-z_3)^a}\;B^{(\alpha,\beta;p)}(b,c-b).
\end{multline*}
\end{theorem}
\begin{proof}
Let us consider the following sequences,
\begin{gather*}
x_k= z_1^{k}\;B_p^{(\alpha,\beta)}(b+k,c-b),\; y_k= z_2^{k}\;  \text{ and }\; m_k=(a)_k\frac{z_3^k}{k!},\;\; k\geq0.
\end{gather*}
It is obvious that
$$\forall k\geq0\;\;\;\;\;0\leq x_k\leq B_p^{(\alpha,\beta)}(b,c-b),\;\; 0\leq y_k\leq 1.$$

We have the following formulas
\begin{gather*}
\sum_{k=0}^{\infty} m_k=\sum_{k=0}^{\infty}(a)_k\frac{z_3^k}{k!}=(1-z_3)^{-a} \text{ and }
\sum_{k=0}^{\infty} m_k\,y_k=\sum_{k=0}^{\infty}(a)_k\,\frac{(z_2z_3)^k}{k!}=(1-z_2z_3)^{-a},
\end{gather*}
and from \eqref{GGHF} we get, respectively,
\begin{gather*}
\sum_{k=0}^{\infty} m_k\,x_k\,y_k= \sum_{k=0}^{\infty} B_p^{(\alpha,\beta)}(b+k,c-b)\,\frac{(z_1z_2z_3)^k}{k!}= B(b,c-b)\,_2F_1^{(\alpha,\beta;p)}\big(a,b;c;z_1z_2z_3\big),
\end{gather*}
\begin{gather*}
\sum_{k=0}^{\infty} m_k\,x_k=\sum_{k=0}^{\infty}(a)_k\,B_p^{(\alpha,\beta)}(b+k,c-b)\,\frac{(z_1z_3)^k}{k!}= B(b,c-b)\,_2F_1^{(\alpha,\beta;p)}\big(a,b;c;z_1z_3\big).
\end{gather*}
As previous, applying the Gr\"{u}ss inequality \eqref{GD}, taking into account these last sums and letting  $n\longrightarrow \infty$, we deduce the desired inequality. The details are similar to those of the previous results and therefore omitted here.
\end{proof}

The case $p=0$ in Theorem \ref{thC} corresponds to the following result.

\begin{corollary}
With the same conditions as in Theorem \ref{thC} we have
\begin{equation*}
\left\vert \,(1-z_2z_3)^a\,_2F_1\big(a,b;c;z_1z_2z_3\big)-(1-z_3)^a\,_2F_1\big(a,b;c;z_1z_3\big) \right \vert
\leq\frac{1}{4}\frac{(1-z_2z_3)^a}{(1-z_3)^a}.
\end{equation*}
\end{corollary}

\section{Concluding remarks}
Studies concerning the hypergeometric functions are increasing, especially to provide generalizations and to establish suitable properties that allow to undertake various calculus related to those special functions. \\
In this paper, we have been interested to this last topic. Namely, we investigated some approximations involving the generalized Gauss hypergeometric and the generalized confluent hypergeometric functions. The use of the discrete Gr\"{u}ss inequality enabled us to establish several analytic inequalities for these two functions.

\end{document}